\documentclass[preprint,11pt]{elsarticle}
\usepackage{mathrsfs}
\usepackage{amsfonts}
\usepackage{amsmath}
\usepackage{amstext}
\usepackage{stmaryrd}
\usepackage{amssymb}
\usepackage{amsthm}
\usepackage{mathrsfs}
\usepackage{amsfonts}
\usepackage{enumerate}
\usepackage{amscd}
\usepackage{indentfirst}
\usepackage{enumerate}
\usepackage{amsmath,amsfonts,amssymb,amsthm}
\usepackage{amsmath,amssymb,amsthm,amscd}
\usepackage{graphicx,mathrsfs}
\usepackage{appendix}

\numberwithin{equation}{section}

\newtheorem{theorem}{Theorem}[section]
\newtheorem{proposition}[theorem]{Proposition}

\newtheorem{lemma}[theorem]{Lemma}

\theoremstyle{definition}

\newtheorem{definition}[theorem]{Definition}

\theoremstyle{remark}
\newtheorem{remark}[theorem]{Remark}

\begin{document}
\begin{frontmatter}

\title{Steady Vortex Patch Solutions to the Vortex-Wave System
}

\author{Daomin Cao}
\ead{dmcao@amt.ac.cn}
\author{Guodong Wang}

\ead{wangguodong14@mails.ucas.ac.cn}

\address{}

\begin{abstract}
The vortex-wave system describes the motion of a two-dimensional ideal fluid in which the vorticity includes continuously distributed vorticity, which is called the background vorticity, and a finite number of concentrated vortices. In this paper we restrict ourselves to the case of a single point vortex in bounded domains. We prove the existence of steady vortex patch solutions to this system with prescribed distribution for the background vorticity. Moreover, we show that the supports of these solutions ``shrink'' to a minimum point of the Kirchhoff-Routh function as the strength parameter of the background vorticity goes to infinity.
\end{abstract}
\end{frontmatter}



\section{Introduction}
The vortex-wave system was firstly introduced by Marchioro and Pulvirenti in \cite{MPu2} to describe the motion of a planar ideal fluid in which the vorticity consists of continuously distributed vorticity(wave part) and $k$ concentrated vortices(vortex part). In the whole plane the system can be written as follows:
\begin{equation}\label{11}
\begin{cases}
 \partial_t{\omega}+\mathbf{u}\cdot\nabla\omega=0,
 \\ \frac{dx_i}{dt}=J\nabla \Gamma*\omega(x_i,t)+\sum_{j\neq i}\kappa_jJ\nabla\Gamma(x_i-x_j),\,\,i=1,\cdot\cdot\cdot,k,
  \\ \mathbf{u}=J\nabla \Gamma*\omega+\sum_{j=1}^k\kappa_jJ\nabla \Gamma(\cdot-x_j),
\end{cases}
\end{equation}
where $\Gamma(x)=-\frac{1}{2\pi}\ln|x|$ is the fundamental solution of $-\Delta$ in $\mathbb{R}^2$, $J(x_1,x_2)=(x_2,-x_1)$ denotes clockwise rotation through $\frac{\pi}{2}$, and $\Gamma*\omega$ is the Newton potential of $\omega$ defined by
\begin{equation}\label{12}
\Gamma*\omega(x,t)=-\frac{1}{2\pi}\int_{\mathbb{R}^2}\ln|x-y|\omega(y,t)dy.
\end{equation}

Let us explain system \eqref{11} briefly. The first equation is a transport equation for the background vorticity $\omega(x,t)$, which means that the background vorticity is transported by the velocity ``generated'' by itself(the term $J\nabla \Gamma*\omega$), and $k$ point vortices(the term $\sum_{j=1}^k\kappa_jJ\nabla \Gamma(\cdot-x_j)$). The second equation means that the evolution of each vortex $x_i(t)$ is influenced by the velocity ``generated'' by the background vorticity(the term $J\nabla \Gamma*\omega(x_i,t)$) and the other $k-1$ vortices(the term $\sum_{j\neq i}\kappa_jJ\nabla\Gamma(x_i-x_j)$). If $\kappa_i=0$, $i=1,\cdot\cdot\cdot,k$, then the system reduces to the vorticity form of the Euler equation, which has been extensively studied, see\cite{MB,MPu,Y} for example. If the background vorticity vanishes, then the system becomes the Kirchhoff-Routh equation, which is a model to describe the motion of $k$ concentrated vortices, see \cite{L,MPa,T3} for example.

The existence and uniqueness to the non-stationary vortex-wave system in the whole plane have been extensively studied over the past decades, see \cite{B,CLMN,LM,LM2,MPu2,Mi} for example. However, as far as we know, little work has been done in steady solutions to this system. Our purpose here is to construct steady vortex patch solutions in the case of a single vortex. More precisely, we will prove that for any vortex patch rearrangement class $\mathcal{N}^\mu$ defined by
\[
\mathcal{N}^\mu=\{\omega\in L^\infty(D)\,\,|\,\,\omega=\mu I_A, A\subset D, \mu|A|=1\},
\]
where $\mu$ is the vorticity strength parameter, there exists a steady solution to the vortex-wave system, say $(\omega^\mu,x^\mu)$, satisfying $\omega^\mu\in\mathcal{N}^\mu$. Moreover, as the strength parameter $\mu$ goes to infinity, both $supp(\omega^\mu)$ and $x^\mu$  "shrink" to a minimum point of the Kirchhoff-Routh function.

The basic idea to prove the existence of $(\omega^\mu,x^\mu)$ for fixed $\mu$ is to construct a family of steady vortex patch solutions to the Euler equation, in which one part of the vorticity belongs to the rearrangement class $\mathcal{N}^\mu$ while the other part ``shrinks'' to a point, then we show the limit is in fact a steady solution to the vortex-wave system. We will use the result of Burton \cite{B4} on maximization of convex functionals on rearrangement class to obtain approximate solutions, while the proof of the convergence is based on the idea of Turkington \cite{T}.

It is worth mentioning that our result is closely related to the desingularization of point vortices for the Euler equation, which has been studied by many authors, see \cite{B,CPY,EM2,M,MPa,T3,W} for example. Roughly speaking, desingularization of vortices for the Euler equation is to justify the Kirchhoff-Routh equation by approximation of the classical Euler equation. There are mainly two kinds of desingularization in the literature: the first kind is to consider a family of initial vorticity, which is sufficiently concentrated in $k$ small regions, then the evolved vorticity according the Euler equation is also concentrated in $k$ small regions for all time, and the limiting positions of these small regions can be approximated by the Kirchhoff-Routh equation, see \cite{M,MPa,T3} and the references therein; the second kind is to construct a sequence of steady solutions to the Euler equation that ``shrinks'' to a critical point of the Kirchhoff-Routh function(or equivalently, a stationary solution to the Kirchhoff-Routh equation), see \cite{CPY,SV,T,W} for example.

Analogously, it is natural to consider the desingularization for the vortex-wave system. In \cite{B}, the author considered the first kind of desingularization, i.e., given a sequence of initial vorticity which is the sum of a given background vorticity and a concentrated vorticity ``blob'', it was proved that the sequence of the evolved solutions according to the Euler equation converges to the vortex-wave system in some sense. In contrast to \cite{B}, in this paper we are concerned with the the second kind of desingularization, i.e., we construct a family of steady Euler solutions in which one part of the vorticity belongs to a given rearrangement class while the support of other part ``shrinks'' to a point, and the limit is exactly a steady solution to the vortex-wave system.

We end this section by giving outline of this paper. In Section 2, we introduce the vortex-wave system in bounded domains and state our main results. Then we devote Section 3 to the construction of approximate solutions by solving a certain variational problem. In Section 4 by comparing energy we show that the limit of approximate solutions is in fact a steady vortex patch solution to the vortex-wave system. In Section 5 we consider the limit of the steady vortex solutions obtained in Section 4 as the strength of the background vorticity goes to infinity.

\section{Main Results}
\subsection{Notations}

Let $D\subset\mathbb{R}^2$ be a bounded and simply-connected domain with smooth boundary. The Green's function for $-\Delta$ in $D$ with zero Dirichlet boundary condition is written as
\begin{equation}\label{21}
G(x,y)=\frac{1}{2\pi}\ln \frac1{|x-y|}-h(x,y), \,\,\,x,y\in
D,
\end{equation}
where $h(x,y)$ is the regular part of $G$. Note that $h(\cdot,\cdot)$ is bounded from below in $D\times D$. The Kirchhoff-Routh function of $D$ is defined to be
\begin{equation}\label{22}
H(x)=\frac{1}{2}h(x,x),\,\,\,x\in D,
\end{equation}
 and $\lim_{x\rightarrow\partial D}H(x)=+\infty$, see \cite{T3}, Lemma 2.2 for example. $2H$ is also called Robin function.

We shall use the following notations throughout this paper: $J(a,b)=(b,-a)$ denotes clockwise rotation through $\frac{\pi}{2}$ for any vector $(a,b)\in \mathbb{R}^2$, $|A|$ denotes the two-dimensional Lebesgue measure for any measurable set $A\subset \mathbb{R}^2$, $\overline{A}$ denotes the closure of some set $A\subset \mathbb{R}^2$ in the Euclidean topology, and $I_A$ denotes the characteristic function of some planar set $A$, that is, $I_A(x)=1$ if $x\in A$, $I_A(x)=0$ elsewhere.  $supp(g)$ denotes the support of some function $g$, that is,
\begin{equation}
supp(g)=\overline{\{x\,|\,g(x)\neq0\}}.
\end{equation}
$dist(\cdot,\cdot)$ denotes the distance between two sets,
\begin{equation}
dist(A,B)=\inf_{x\in A,y\in B}|x-y|.
\end{equation}
For a given measurable function $g$ on $D$, the rearrangement class of $g$ is defined by
\begin{equation}
\mathcal{R}(g)=\{f \text{ is measurable}\,\,|\,\,\text{for any}\,a\in \mathbb{R},\,\,|\{f>a\}|=|\{g>a\}| \}.
\end{equation}

For any $\omega\in L^\infty(D)$, we also define the stream function of $\omega$ by
\begin{equation}
G*\omega(x)=\int_DG(x,y)\omega(y)dy.
\end{equation}
Note that since $\omega\in L^p(D)$ for any $p\in[1,+\infty]$, by $L^p$ estimate and Sobolev embedding $G*\omega\in W^{2,p}(D)\cap C^{1,\alpha}(\overline{D})$ for any $p\in[1,+\infty)$ and $\alpha\in(0,1)$.
\subsection{The vortex-wave system in bounded domains}
We begin with a discussion on the Euler equation describing an ideal fluid with unit density moving in $D$,
\begin{equation}\label{280}
\begin{cases}
 \partial_t\mathbf{u}+(\mathbf{u}\cdot\nabla)\mathbf{u}=-\nabla P,\\
 \nabla\cdot\mathbf{u}=0,\\
 \mathbf{u}\cdot\mathbf{n}|_{\partial D}=0,
 \end{cases}
\end{equation}
where $\mathbf{u}=(u_1,u_2)$ is the velocity field, $P$ is the pressure, and $\mathbf{n}$ is the outward unit normal. Here we impose the impermeability boundary condition.

Define the vorticity $\omega=\partial_1 u_2-\partial_2u_1$. Since $D$ is simply connected, $\mathbf{u}$ can be  uniquely determined by $\omega$,
\begin{equation}\mathbf{u}=J\nabla G*\omega,
\end{equation}
see \cite{MPu}, \S1.2 for example. So it suffices to consider the equation satisfied by $\omega$.
Using the identity $\frac{1}{2}\nabla|\mathbf{u}|^2=(\mathbf{u}\cdot\nabla)\mathbf{u}+\omega J\mathbf{u}$, the first equation of $\eqref{280}$ becomes
\begin{equation}\label{281}
 \mathbf{u}_t+\nabla(\frac{1}{2}|\mathbf{u}|^2+P)-\omega J\mathbf{u}=0,
\end{equation}
Taking the curl on both sides we obtain the vorticity form of the Euler equation

\begin{equation}\label{282}
 \omega_t+\mathbf{u}\cdot\nabla\omega=0,
\end{equation}
 which means that the vorticity is transport by the velocity $\mathbf{u}$, where $\mathbf{u}$ is "generated" by $\omega$, i.e., $\mathbf{u}=J\nabla G*\omega$.

When the vorticity is sufficiently concentrated at $k$ points, equation \eqref{282} is approximated by the following Kirchhoff-Routh equation:
\begin{equation}\label{283}
\frac{dx_i}{dt}=\sum_{j=1,j\neq i}^ka_jJ\nabla_{x_i}G(x_i,x_j)-a_iJ\nabla H(x_i),\,\,i=1,\cdot\cdot\cdot,k,
\end{equation}
where $x_i(t)$ represents the position of the $i$-th vortex, and $a_i$ is the corresponding vorticity strength. Equation \eqref{283} means that each vortex interacts with the others via the term $a_jJ\nabla_{x_i}G(x_i,x_j)$ and with the boundary via the term $-a_iJ\nabla H(x_i)$. The approximation from the Euler equation to the Kirchhoff-Routh equation has been extensively studied, see \cite{CPY,EM2,M,MPa,T3} and the references therein.

Now we combine the Euler equation and the Kirchhoff-Routh equation together, that is, we assume that the vorticity consists of both continuously distributed vorticity denoted by $\omega(x,t)$ and $k$ concentrated vortices $x_i(t)$, $i=1,\cdot\cdot\cdot,k$. Then it is reasonable that the evolution of $\omega(x,t)$ and $x_i(t)$ obey the following system:
\begin{equation}\label{23}
\begin{cases}
 \partial_t{\omega}+\mathbf{u}\cdot\nabla\omega=0,
 \\ \frac{dx_i}{dt}=J\nabla G*\omega(x_i,t)+\sum_{j=1,j\neq i}^ka_jJ\nabla_{x_i}G(x_i,x_j)-a_iJ\nabla H(x_i),
  \\ \mathbf{u}=J\nabla G*\omega+\sum_{j=1}^ka_jJ\nabla G(x_j,\cdot),
\end{cases}
\end{equation}
which we call the vortex-wave system in bounded domains.

Let us explain \eqref{23} briefly. The first equation in \eqref{23} means that evolution of the background vorticity $\omega(x,t)$ is influenced by the velocity field $J\nabla G*\omega$ ``generated'' by itself and the velocity field $\sum_{j=1}^ka_jJ\nabla G(x_j,\cdot)$ ``generated'' by the $k$ point vortices with strength $a_i$, and the evolution of each $x_i(t)$ is influenced by the velocity field $J\nabla G*\omega(x_i,t)$ ``generated'' by $\omega$ and the velocity $\sum_{j=1,j\neq i}^ka_jJ\nabla_{x_i}G(x_i,x_j)$ ``generated'' by the other $k-1$ point vortices together with the boundary term $-a_i J\nabla H(x)$. If $\omega\equiv 0$, then \eqref{23} is exactly the Kirchhoff-Routh equation; if $a_i=0$, $i=1,\cdots,k$, then $\eqref{23}$ becomes the vorticity form of the Euler equation.

\subsection{Main results}

In the rest of this paper we will restrict ourselves to the stationary vortex-wave system with a single point vortex(that is $k=1$), and  we assume that the point vortex has unit strength for simplicity (namely $a_1=1$).

More precisely, we will consider the following system:
\begin{equation}\label{24}
\begin{cases}
J\nabla (G*\omega+ G(x,\cdot))\cdot\nabla\omega=0,
 \\ \nabla G*\omega(x)- \nabla H(x)=0.
\end{cases}
\end{equation}

Since we are going to deal with vortex patch solutions which are discontinuous, it is necessary to introduce the weak formulation for the first equation in \eqref{24}. To motivate the definition, let us assume that $\omega$ is a smooth solution, then for any $\phi\in C_c^\infty(D)$,

\begin{equation}
\int_D\phi J\nabla (G*\omega+ G(x,\cdot))\cdot\nabla\omega dy=0.
\end{equation}
Now we claim that
\begin{equation}\label{2018}
\int_D\phi J\nabla (G*\omega+ G(x,\cdot))\cdot\nabla\omega dy=-\int_D\omega J\nabla(G*\omega+ G(x,\cdot))\cdot \nabla\phi dy.
\end{equation}
In fact, by the divergence theorem
\begin{equation}
\begin{split}
\int_D\phi J\nabla G*\omega\cdot\nabla\omega dy=&\int_D\phi\, div(\omega J\nabla G*\omega)dy\\
=&\int_D\,div(\phi\omega J\nabla G*\omega)dy-\int_D\omega J\nabla G*\omega\cdot\nabla\phi dy\\
=&\int_{\partial D}\phi\omega J\nabla G*\omega\cdot\mathbf{n}dS-\int_D\omega J\nabla G*\omega\cdot\nabla\phi dy\\
=&-\int_D\omega J\nabla G*\omega\cdot\nabla\phi dy,
\end{split}
\end{equation}
where we use the fact that $J\nabla G*\omega$ is a divergence-free vector field. To calculate the integral $\int_D\phi J\nabla G(x,\cdot)\cdot\nabla\omega dy$, the singularity of $\nabla G$ need to to be dealt with. To this end define $\Omega^a=\{y\in D\,|\,G(x,y)>a\}$
and $D^a=D\setminus \overline{\Omega^a}$. By the implicit function theorem, $\Omega^a$ is a simply connected domain with smooth boundary if $a>0$ is sufficiently large. Again by the divergence theorem
\begin{equation}
\begin{split}
\int_{D^a}\phi J\nabla G(x,\cdot)\cdot\nabla\omega dy=&\int_{D^a}\phi\, div(\omega J\nabla G(x,\cdot))dy\\
=&\int_{D^a}\,div(\phi\omega J\nabla G(x,\cdot))dy-\int_{D^a}\omega J\nabla G(x,\cdot)\cdot\nabla\phi dy\\
=&\int_{\partial {D^a}}\phi\omega J\nabla G(x,\cdot)\cdot\mathbf{n}dS-\int_{D^a}\omega J\nabla G(x,\cdot)\cdot\nabla\phi dy\\
=&-\int_{D^a}\omega J\nabla G(x,\cdot)\cdot\nabla\phi dy.
\end{split}
\end{equation}
On the other hand, by Lebesgue's dominated convergence theorem(notice that $\nabla G(x,\cdot)\in L^1(D)$) we have
\begin{equation}
\lim_{a\rightarrow+\infty}\int_{D^a}\phi J\nabla G(x,\cdot)\cdot\nabla\omega dy=\int_{D}\phi J\nabla G(x,\cdot)\cdot\nabla\omega dy,
\end{equation}
and
\begin{equation}
\lim_{a\rightarrow+\infty}\int_{D^a}\omega J\nabla G(x,\cdot)\cdot\nabla\phi dy=\int_{D}\omega J\nabla G(x,\cdot)\cdot\nabla\phi dy.
\end{equation}
Taking the limit we obtain
\begin{equation}
\int_{D}\phi J\nabla G(x,\cdot)\cdot\nabla\omega dy=-\int_{D}\omega J\nabla G(x,\cdot)\cdot\nabla\phi dy.
\end{equation}
Hence we have proved \eqref{2018}.
In conclusion, if $\omega$ is a smooth solution to the system \eqref{24}, then it must satisfy
\begin{equation}\label{201}
\int_D\omega J\nabla(G*\omega+ G(x,\cdot))\cdot \nabla\phi dy=0.
\end{equation}

Notice that the integral in \eqref{201} makes sense for any $\omega\in L^\infty(D)$ since $G*\omega\in C^{1}(\overline{D})$ and $\nabla G(x,\cdot)\in L^1(D)$, so we have the following definition:

\begin{definition}
$(\omega,x)$ is called a weak solution to \eqref{24} if $\omega\in L^\infty(D), x\in D$ and
\begin{equation}\label{25}
\begin{cases}
\int_D\omega J\nabla(G*\omega+ G(x,\cdot))\cdot \nabla\phi dy=0, \,\,\forall \phi\in C_c^\infty(D)\\
\,\\
\nabla G*\omega(x)- \nabla H(x)=0.
\end{cases}
\end{equation}
\end{definition}

In this paper we are mainly interested in the vortex patch solution of \eqref{24}, i.e., the solution $(\omega, x)$ such that $\omega$ is of the form $\omega=a I_A$, where $a$ is a real number representing the strength of $\omega$ and $A\subset D$ is a Lebesgue measurable set.

The main result of this paper is as follows:

\begin{theorem}\label{29}
Let $\mu$ be a positive real number satisfying $\mu>\frac{1}{|D|}$, and $\mathcal{N}^\mu$ be a rearrangement class defined by
\begin{equation}\label{210}
\mathcal{N}^\mu=\{\omega\in L^\infty(D)\,\,|\,\,\omega=\mu I_A, A \,\text{is a measurable set in}\,D, \mu|A|=1\}.
\end{equation}
Then there exist $\omega^\mu\in \mathcal{N}^\mu$ and $x^\mu\in D$ such that $(\omega^\mu,x^\mu)$ is a weak solution to the stationary vortex-wave system \eqref{24}, moreover, $\omega^\mu$ has the form
\begin{equation}\label{211}
\omega^\mu=\mu I_{\{G*\omega^\mu+G(x^\mu,\cdot)>b^\mu\}}
\end{equation}
for some $b^\mu>0$.
\end{theorem}

\begin{remark}
If $\mu=\frac{1}{|D|}$, then there is only one element in $\mathcal{N}^\mu$, that is $\omega\equiv\mu$. In this case $\omega$ is a smooth function and the first equation in \eqref{24} is satisfied for any $x\in D$, so we need only consider the second equation. Notice that $H|_{\partial D}=+\infty$, so we can always choose $x\in D$ such that $x$ is a maximum point, thus a critical point, for the function $G*\omega-H$.
\end{remark}

 As for the asymptotic behavior of $(\omega^\mu,x^\mu)$ as $\mu\rightarrow+\infty$, we can prove that up to a subsequence ``most part'' of $\omega^\mu$ concentrates near a minimum point of $H$, say $x^*$, and at the same time $x^\mu\rightarrow x^*$.
\begin{theorem}\label{90}
Let $(\omega^\mu,x^\mu)$ be the weak solution to the stationary vortex-wave system \eqref{24} obtained in Theorem \ref{29}, then up to a subsequence we have $x^\mu\rightarrow x^*$ as $\mu\rightarrow+\infty$, where $x^*$ is a minimum point of $H$. Moreover, there exists $r^\mu$, $r^\mu\rightarrow0$ as $\mu\rightarrow +\infty$, such that
 \begin{equation}\label{10022}
 \lim_{\mu\rightarrow +\infty}\int_{B_{r^\mu}(x^*)}\omega^\mu(x)dx=1.
 \end{equation}
\end{theorem}
\begin{remark}\label{457}
Recalling that $\int_{D}\omega^\mu(x)dx=1$, it is easy to check that $\omega^\mu\rightarrow \delta(x^*)$ as $\mu\rightarrow +\infty$ in the distributional sense, where $\delta(x^*)$ is the Dirac measure located at $x^*$. More precisely,
\begin{equation}
\lim_{\mu\rightarrow+\infty}\int_D\omega^\mu(x)\phi(x)dx=\phi(x^*),\,\,\forall \phi\in C_c^\infty(D).
\end{equation}
In fact,
\begin{equation}
\begin{split}
\left|\int_D\omega^\mu(x)\phi(x)dx-\phi(x^*)\right|=&\left|\int_D(\phi(x)-\phi(x^*))\omega^\mu(x)dx\right|\\
\leq& \left|\int_{B_{r^\mu}(x^*)}(\phi(x)-\phi(x^*))\omega^\mu(x)dx\right|+\left|\int_{D\setminus B_{r^\mu}(x^*)}(\phi(x)-\phi(x^*))\omega^\mu(x)dx\right|\\
\leq&\sup_{x\in B_{r^\mu}(x^*)}|\phi(x)-\phi(x^*)|+2\sup_{x\in D}|\phi(x)|\int_{D\setminus B_{r^\mu}(x^*)}\omega^\mu(x)dx
\end{split}
\end{equation}
which goes to 0 as $\mu\rightarrow +\infty$, where we use \eqref{10022} and the continuity of $\phi$ at $x^*$.
\end{remark}
\begin{remark}
When $D$ is convex, $H$ is a strictly convex function (see \cite{CF1}), so there is only one minimum point for $H$. In this case the phrase ``up to a subsequence'' in Theorem \ref{90} can been removed.
\end{remark}

\section{Variational Problem}
Throughout this section we assume that $\mu$ is a fixed positive real number. We will construct a family of steady vortex patch solutions to the Euler equation and analyze their properties.

Let $\lambda$ be a positive number. Define
\begin{equation}\label{31}
\mathcal{M}^\lambda=\{\omega\in L^\infty(D)\,\,|\,\,\omega=\omega_1+\omega_2, \omega_1\in\mathcal{N}^\mu, \omega_2=\lambda I_B, \lambda|B|=1, supp(\omega_1)\cap B=\varnothing\}.
\end{equation}
Recall that $\mathcal{N}^\mu$ is defined in Theorem \ref{29}.
 For sufficiently large $\lambda$, since $\mu>\frac{1}{|D|}$, we know that $\mathcal{M}^\lambda$ is not empty. Moreover, it is easy to check that $\mathcal{M}^\lambda$ is a rearrangement class of any element in it if $\lambda>\mu$, that is, for any $\omega\in \mathcal{M}^\lambda$ we have $\mathcal{M}^\lambda=\mathcal{R}(\omega)$.
In the following we always assume $\lambda$ to be sufficiently large.

Now define the energy functional on $\mathcal{M}^\lambda$ by

\begin{equation}\label{32}
E(\omega)=\frac{1}{2}\int_D\int_DG(x,y)\omega(x)\omega(y)dxdy,\,\, \omega\in \mathcal{M}^\lambda,
\end{equation}
which represents the kinetic energy of an ideal fluid in $D$ with vorticity $\omega$.

Existence of a maximizer for $E$ relative to $\mathcal{M}^\lambda$ is an easy consequence of Corollary 3.4 in \cite{B4}. Therein by choosing $\mathcal{L}=-\Delta$, $E=\Psi$, $\mathcal{F}=\mathcal{M}^\lambda$ and $K$ as the Green's operator, we have:
\begin{proposition}\label{33}
There exists a maximizer for E relative to $\mathcal{M}^\lambda$; moreover, if $\omega^\lambda$ is a maximizer, then $\omega^\lambda=f(G*\omega^\lambda)$ a.e. in $D$ for some increasing function $f:\mathbb{R}\rightarrow\mathbb{R}$.
\end{proposition}
\begin{remark}
$\omega^\lambda$ is in fact a steady weak solution to the Euler equation, we refer the interested reader to \cite{T} for a simple proof.
\end{remark}

Let $\omega^\lambda\in \mathcal{M}^\lambda$ be a maximizer, then we can write $\omega^\lambda=\omega_1^\lambda+\omega^\lambda_2$, where $\omega_1^\lambda\in\mathcal{N}^\mu, \omega_2^\lambda=\lambda I_{B^\lambda}, \lambda|B^\lambda|=1,$ and $supp(\omega_1^\lambda)\cap B^\lambda=\varnothing$. For convenience we shall write $\psi^\lambda=G*\omega^\lambda$ and $\psi_i^\lambda=G*\omega_i^\lambda$, $i=1,2.$

\begin{lemma}\label{34}
$\omega^\lambda_2=\lambda I_{\{\psi^\lambda>c^\lambda\}}$ for some $c^\lambda>0$.
\end{lemma}

\begin{proof}
Since $|\{\omega^\lambda_2=\lambda\}|>0$ and $\omega^\lambda=f(\psi^\lambda)$ a.e. in $D$, it follows that $\{t\in\mathbb{R}\,|\,f(t)=\lambda\}$ is not empty, then we can define $c^\lambda=inf\{t\in\mathbb{R}\,|\,f(t)=\lambda\}$. By the fact that $f$ is an increasing function and $\psi^\lambda> 0$ in ${D}$(by strong maximum principle), we have $c^\lambda>0$.

 By the definition of $c^\lambda$, $\omega^\lambda=f(\psi^\lambda)\equiv\lambda$ a.e. on $\{x\in D\,|\,\psi^\lambda(x)>c^\lambda\}$, and $\omega^\lambda<\lambda$ a.e. on $\{x\in D\,|\,\psi^\lambda(x)<c^\lambda\}$. On the set $\{x\in D\,|\,\psi^\lambda(x)=\lambda\}$, we have $\nabla\psi^\lambda\equiv0$ a.e., which implies $\omega^\lambda=-\Delta\psi^\lambda\equiv0$ a.e..

 In conclusion, we have proved that $\{x\in D\,|\,\omega^\lambda(x)=\lambda\}=\{x\in D\,|\,\psi^\lambda(x)>c^\lambda\}$, then by choosing $\lambda>a$ we have $B^\lambda=\{x\in D\,|\,\psi^\lambda(x)>c^\lambda\}$, which is the desired result.
\end{proof}

Now we begin to analyze the asymptotic behavior of $\omega^\lambda_2$ as $\lambda\rightarrow+\infty$. In this and the next section we shall use $C$ to denote various constants not depending on $\lambda$.
\begin{lemma}\label{35}
$E(\omega^\lambda)\geq -\frac{1}{4\pi}\ln\varepsilon-C$, where $\varepsilon$ satisfies $\lambda\pi\varepsilon^2=1$.
\end{lemma}

\begin{proof}

We take the test function as follows: for any fixed $x_1\in D$, define $\bar{\omega}^\lambda=\bar{\omega}^\lambda_1+\bar{\omega}_2^\lambda$, where $\bar{\omega}_2^\lambda=\lambda I_{B_\varepsilon(x_1)}$, $\bar{\omega}^\lambda_1\in \mathcal{N}^\mu$ and $\bar{\omega}^\lambda_1=0$ a.e. in $B_\varepsilon(x_1)$. It's easy to check that $\bar{\omega}^\lambda\in \mathcal{M}^\lambda$, so we have $E(\omega^\lambda)\geq E(\bar{\omega}^\lambda)$. By simple calculation,

\begin{equation}\label{333}
\begin{split}
E(\bar{\omega}^\lambda)=&\frac{1}{2}\int_D\int_DG(x,y)\bar{\omega}^\lambda(x)\bar{\omega}^\lambda(y)dxdy\\
=&\frac{1}{2}\int_D\int_DG(x,y)(\bar{\omega}_1^\lambda(x)+\bar{\omega}_2^\lambda(x))(\bar{\omega}_1^\lambda(y)+\bar{\omega}_2^\lambda(y))dxdy\\
=& E(\bar{\omega}^\lambda_1)+E(\bar{\omega}^\lambda_2)+\frac{1}{2}\int_D\int_DG(x,y)\bar{\omega}^\lambda_1(x)\bar{\omega}^\lambda_2(y)dxdy+\frac{1}{2}\int_D\int_DG(x,y)\bar{\omega}^\lambda_2(x)\bar{\omega}^\lambda_1(y)dxdy \\
=& E(\bar{\omega}^\lambda_1)+E(\bar{\omega}^\lambda_2)+\int_D\int_DG(x,y)\bar{\omega}^\lambda_1(x)\bar{\omega}^\lambda_2(y)dxdy,
\end{split}
\end{equation}
where we use the symmetry of the Green's function, that is, $G(x,y)=G(y,x)$ for any $x,y\in D$.

Since $G\in L^1(D\times D)$, we have the following estimate for $E(\bar{\omega}^\lambda_1)$:
\begin{equation}\label{334}
|E(\bar{\omega}_1^\lambda)|=\left|\frac{1}{2}\int_D\int_DG(x,y)\bar{\omega}_1^\lambda(x)\bar{\omega}_1^\lambda(y)dxdy\right|\leq \frac{\mu^2}{2}\left|\int_D\int_DG(x,y)dxdy\right|\leq C.
\end{equation}

For the term $\int_D\int_DG(x,y)\bar{\omega}^\lambda_1(x)\bar{\omega}^\lambda_2(y)dxdy$ in \eqref{333}, by $L^p$ estimate we have
\begin{equation}\label{335}
\left|\int_D\int_DG(x,y)\bar{\omega}^\lambda_1(x)\bar{\omega}^\lambda_2(y)dxdy\right|=\left|\int_D G*\bar{\omega}^\lambda_1(y)\bar{\omega}^\lambda_2(y)dy\right|\leq C \int_D \bar{\omega}^\lambda_2(y)dy=C.
\end{equation}

It remains to estimate the lower bound of $E(\bar{\omega}^\lambda_2)$,
\begin{equation}\label{336}
\begin{split}
E(\bar{\omega}^\lambda_2)&=\frac{1}{2}\int_D\int_DG(x,y)\bar{\omega}^\lambda_2(x)\bar{\omega}^\lambda_2(y)dxdy\\
&=-\frac{1}{4\pi}\int_D\int_D\ln|x-y|\bar{\omega}^\lambda_2(x)\bar{\omega}^\lambda_2(y)dxdy-\frac{1}{2}\int_D\int_Dh(x,y)\bar{\omega}^\lambda_2(x)\bar{\omega}^\lambda_2(y)dxdy\\
&=-\frac{\lambda^2}{4\pi}\int_{B_\varepsilon(x_1)}\int_{B_\varepsilon(x_1)}\ln|x-y|dxdy-\frac{1}{2}\int_D\int_Dh(x,y)\bar{\omega}^\lambda_2(x)\bar{\omega}^\lambda_2(y)dxdy,
\end{split}
\end{equation}
Since $|x-y|\leq2\varepsilon$ for $x,y\in B_{\varepsilon}(x_1)$, we have
\[\begin{split}
-\frac{\lambda^2}{4\pi}\int_{B_\varepsilon(x_1)}\int_{B_\varepsilon(x_1)}\ln|x-y|dxdy
\geq& -\frac{\lambda^2}{4\pi}\int_{B_\varepsilon(x_1)}\int_{B_\varepsilon(x_1)}\ln|2\varepsilon|dxdy\\
=&-\frac{1}{4\pi}\ln\varepsilon-\frac{1}{4\pi}\ln2.
\end{split}\]
On the other hand, by the continuity of $h(x,y)$ in $D\times D$, the integral $\int_{B_\varepsilon(x_1)}\int_{B_\varepsilon(x_1)}h(x,y)dxdy$ converges to $h(x_1,x_1)$, thus is uniformly bounded, as $\lambda \rightarrow +\infty$, so
\begin{equation}\label{337}
E(\bar{\omega}^\lambda_2)\geq-\frac{1}{4\pi}\ln \varepsilon-C.
\end{equation}

Using \eqref{333},\eqref{334},\eqref{335} and \eqref{337} we complete the proof.
\end{proof}

Now we define $T(\omega^\lambda)=\frac{1}{2}\int_D\omega_2^\lambda(x)(\psi^\lambda-c^\lambda)(x)dx$, which represents the kinetic energy of the fluid on $B^\lambda$. To simplify presentation we write $\zeta^\lambda=\psi^\lambda-c^\lambda$. By the fact that $\zeta^\lambda=0$ on $\partial B^\lambda$, so
\begin{equation}\label{381}
T(\omega^\lambda)=\frac{1}{2}\int_{B^\lambda}\omega_2^\lambda(x)\zeta^\lambda(x)dx=
\frac{1}{2}\int_{B^\lambda}|\nabla\zeta^\lambda(x)|^2dx.
\end{equation}

 We have the following uniform estimate for $T$:
\begin{lemma}\label{36}
$T(\omega^\lambda)\leq C.$
\end{lemma}

\begin{proof}

Firstly by H\"{o}lder's inequality, we have
\[\begin{split}
T(\omega^\lambda)=\frac{1}{2}\lambda\int_{B^\lambda}\zeta^\lambda(x)dx\leq\frac{1}{2}\lambda|B^\lambda|^{\frac{1}{2}}\{\int_{B^\lambda}|\zeta^\lambda(x)|^2dx\}^{\frac{1}{2}}.
\end{split}\]
By the Sobolev embedding $W_0^{1,1}(D)\hookrightarrow L^2(D)$, we have
\[\begin{split}
\left\{\int_{B^\lambda}|\zeta^\lambda(x)|^2dx\right\}^{\frac{1}{2}}=\left\{\int_{D}|(\zeta^\lambda)^+(x)|^2dx\right\}^{\frac{1}{2}}\leq& C\int_{D}|\nabla(\zeta^\lambda)^+(x)|dx,
\end{split}\]
where $(\zeta^\lambda)^+(x)=max\{0,\zeta^\lambda(x)\}$.
It follows that
\[\begin{split}
T(\omega_\lambda)\leq&C\lambda|B^\lambda|^{\frac{1}{2}}\int_{D}|\nabla(\zeta^\lambda)^+(x)|dx=C\lambda|B^\lambda|^{\frac{1}{2}}\int_{B^\lambda}|\nabla\zeta^\lambda(x)|dx\leq C\lambda|B^\lambda|\left\{\int_{B^\lambda}|\nabla\zeta^\lambda(x)|^2dx\right\}^{\frac{1}{2}}.
\end{split}\]
Notice that $\lambda|B^\lambda|=\int_D\omega^\lambda_2(x)dx=1$, we obtain
\begin{equation}\label{382}
T(\omega^\lambda)\leq C\left\{\int_{B^\lambda}|\nabla\zeta^\lambda(x)|^2dx\right\}^{\frac{1}{2}}.
\end{equation}
By comparing \eqref{382} with \eqref{381} we get the desired result.

\end{proof}

\begin{lemma}\label{37}
There exists $R_0>0$ such that $diam(supp(\omega^\lambda_2))\leq R_0\varepsilon$.
\end{lemma}
\begin{proof}
Firstly we estimate the lower bound for $c^\lambda$. By the definition of $T(\omega^\lambda)$,
\begin{equation}\label{391}
E(\omega^\lambda)=T(\omega^\lambda)+\frac{1}{2}\int_D\omega^\lambda_1(x)\psi^\lambda(x)dx+\frac{c^\lambda}{2},
\end{equation}
It is easy to check that $\int_D\omega^\lambda_1(x)\psi^\lambda(x)dx$ has a uniform upper bounded. In fact,
\begin{equation}
\begin{split}
\int_D\omega^\lambda_1(x)\psi^\lambda(x)dx=&\int_D\omega^\lambda_1(x)G*(\omega^\lambda_1+\omega^\lambda_2)(x)dx\\
=&\int_D\int_DG(x,y)\omega^\lambda_1(x)\omega^\lambda_1(y)dxdy+\int_D\int_DG(x,y)\omega^\lambda_1(x)\omega^\lambda_2(y)dxdy\\
\leq&\mu^2\int_D\int_D|G(x,y)|dxdy+|G*\omega^\lambda_1|_{L^\infty(D)}\\
\leq&C.
\end{split}
\end{equation}
Now \eqref{391} together with Lemma \ref{35} and Lemma \ref{36} gives
\begin{equation}\label{392}
c^\lambda\geq-\frac{1}{2\pi}\ln\varepsilon-C.
\end{equation}

Now for any $x\in supp(\omega^\lambda_2)$, we have $\psi^\lambda(x)\geq c^\lambda$, that is,
\begin{equation}\label{48}
\int_DG(x,y)w^\lambda(y)dy\geq-\frac{1}{2\pi}\ln \varepsilon -C.
\end{equation}
Since $h(x,y)$ is bounded from below on $D\times D$, we have

\begin{equation}\label{49}
\int_D\ln\frac{1}{|x-y|}\omega^\lambda(y)dy+\ln\varepsilon\geq -C,
\end{equation}
or equivalently,
\begin{equation}
\int_D\ln\frac{1}{|x-y|}\omega_1^\lambda(y)dy+\int_D\ln\frac{1}{|x-y|}\omega_2^\lambda(y)dy+\ln\varepsilon\geq -C,
\end{equation}
Notice that
\begin{equation}\label{395}
\left|\int_D\ln\frac{1}{|x-y|}\omega^\lambda_1(y)dy\right|\leq \mu\sup_{x\in D}\left|\int_D\ln|x-y|dy\right|\leq C,
\end{equation}
so we get
\begin{equation}\label{396}
\int_D\ln\frac{\varepsilon}{|x-y|}\omega_2^\lambda(y)dy\geq -C.
\end{equation}

Now let $R>1$ be a positive number to be determined. We divide the integral in \eqref{396} into two parts,
\begin{equation}\label{50}
\int_{B_{R\varepsilon}(x)}\ln\frac{\varepsilon}{|x-y|}\omega_2^\lambda(y)dy+\int_{D\setminus B_{R\varepsilon}(x)}\ln\frac{\varepsilon}{|x-y|}\omega_2^\lambda(y)dy\geq -C.
\end{equation}
By the rearrangement inequality the first integral in $\eqref{50}$ can be estimated as follows:
\begin{equation}\label{1080}
\int_{B_{R\varepsilon}(x)}\ln\frac{\varepsilon}{|x-y|}\omega_2^\lambda(y)dy\leq \lambda\int_{B_\varepsilon(x)}\ln\frac{\varepsilon}{|x-y|}dy=\lambda\int_{B_\varepsilon(0)}\ln\frac{\varepsilon}{|y|}dy=\frac{1}{2}.
\end{equation}
By comparing \eqref{50} with \eqref{1080} we obtain
\[\int_{D\verb|\|B_{R\varepsilon}(x)}\ln\frac{\varepsilon}{|x-y|}\omega_2^\lambda(y)dy \geq -C.\]
We observe now that
\begin{equation}
\int_{D\verb|\|B_{R\varepsilon}(x)}\ln\frac{\varepsilon}{|x-y|}\omega_2^\lambda(y)dy\leq\int_{D\verb|\|B_{R\varepsilon}(x)}\ln\frac{1}{R}\omega_2^\lambda(y)dy,
\end{equation}
therefore
\begin{equation}
\int_{D\verb|\|B_{R\varepsilon}(x)}\omega_2^\lambda(y)dy\leq  \frac{C}{\ln R},
\end{equation}
which means
\begin{equation}
\int_{B_{R\varepsilon}(x)}\omega_2^\lambda(y)dy\geq 1-  \frac{C}{\ln R}.
\end{equation}
Choosing $R$ large such that $ 1- \frac{C}{\ln R}>\frac{1}{2}$, we have
\begin{equation}
\int_{B_{R\varepsilon}(x)}\omega_2^\lambda(y)dy> \frac{1}{2}.
\end{equation}
Since $x\in supp(\omega^\lambda_2)$ is arbitrary and $\int_D\omega^\lambda_2(y)dy=1$, we get the desired result by choosing $R_0=2R$.

\end{proof}

Up to now we have established a family of functions $\omega^\lambda_1$ and $\omega^\lambda_2$, moreover, we show that $diam(supp(\omega^\lambda_2))\rightarrow0$ as $\lambda\rightarrow+\infty$. Now we are in a position to consider the limits of $\omega^\lambda_1$ and $\omega^\lambda_2$. To this end, define the center of $\omega^\lambda_2$ by
\begin{equation}
x^\lambda=\int_Dx\omega^\lambda_2(x)dx.
\end{equation}

Up to a subsequence, we can assume that as $\lambda\rightarrow +\infty$, there exists $x^\mu\in \overline{D}$ such that
$$x^{\lambda}\rightarrow x^\mu.$$

On the other hand, since $\{\omega^\lambda_1\}$ is bounded in $\L^\infty(D)$(recall that $\mu$ is fixed in this section), up to a subsequence we assume that as $\lambda\rightarrow +\infty$
$$\omega_1^\lambda\rightarrow \omega^\mu\,\,\text{ weakly star in}\, L^\infty(D)$$
 for some $\omega^\mu\in \overline{\mathcal{N}^\mu}$, where $\overline{\mathcal{N}^\mu}$ denotes the weak star closure of $\mathcal{N}^\mu$ in $L^\infty(D)$. By standard elliptic equation theory we also have as $\lambda\rightarrow +\infty$
 $$G*\omega^\lambda_1\rightarrow G*\omega^\mu \,\,\text{in}\,C^{1,\alpha}(\overline{D}).$$

We end this section by showing the following lemma which will be frequently used in the next section.

\begin{lemma}\label{700}
We have

(1), $|G*\omega^\lambda_1|_{L^\infty(D)}\leq C$, for some $C>0$ not depending on $\lambda$.

(2), $E(\omega^\lambda_1)=E(\omega^\mu)+o(1)$,

(3), $\int_DG*\omega^\lambda_1(x)\omega^\lambda_2(x)dx=G*\omega^\mu(x^\mu)+o(1)$,

where $o(1)$ denotes quantities such that $o(1)\rightarrow 0$ as $\lambda\rightarrow+\infty$.
\end{lemma}

\begin{proof}
 To prove $(1)$, it suffices to notice that $\omega^\lambda_1$ is bounded in $L^\infty(D)$, then the result follows from $L^p$ estimate and Sobolev embedding.

 Now we turn to the proof of $(2)$. By simple calculation,
\begin{equation}
\begin{split}
\left|E(\omega^\lambda_1)-E(\omega^\mu)\right|=&\left|\frac{1}{2}\int_D\omega^\lambda_1G*\omega^\lambda_1dx- \frac{1}{2}\int_D\omega^\mu G*\omega^\mu dx\right|\\
\leq&\frac{1}{2}\left|\int_D\omega^\lambda_1(G*\omega^\lambda_1-G*\omega^\mu)dx\right|+
\frac{1}{2}\left|\int_DG*\omega^\mu(\omega^\lambda_1-\omega^\mu)dx\right|\\
\leq&\frac{\mu}{2}\left|G*\omega^\lambda_1-G*\omega^\mu\right|_{L^\infty(D)}+o(1),
\end{split}
\end{equation}
which goes to 0 as $\lambda\rightarrow+\infty.$

To prove $(2)$, noting that $diam(supp(\omega^\lambda_2))\rightarrow0$ and $x^\lambda\rightarrow x^\mu$, then we can choose $r^\lambda$, $r^\lambda\rightarrow0$ as $\lambda\rightarrow+\infty$, such that $supp(\omega^\lambda_2)\subset B_{r^\lambda}(x^\mu)$. By the continuity of $G*\omega^\mu$ and the fact $G*\omega^\lambda_1\rightarrow G*\omega^\mu$ in $L^\infty(D)$, it follows that
\begin{equation}
\begin{split}
&\left|\int_DG*\omega^\lambda_1(x)\omega^\lambda_2(x)dx-G*\omega^\mu(x^\mu)\right|\\
=&\left|\int_D(G*\omega^\lambda_1(x)-G*\omega^\mu(x^\mu))\omega^\lambda_2(x)dx\right|\\
=&\left|\int_{B_{r^\lambda}(x^\mu)}(G*\omega^\lambda_1(x)-G*\omega^\mu(x^\mu))\omega^\lambda_2(x)dx\right|\\
\leq&\sup_{x\in B_{r^\lambda}(x^\mu)}|G*\omega^\lambda_1(x)-G*\omega^\mu(x^\mu)|\\
\leq&\sup_{x\in B_{r^\lambda}(x^\mu)}|G*\omega^\lambda_1(x)-G*\omega^\mu(x)|+\sup_{x\in B_{r^\lambda}(x^\mu)}|G*\omega^\mu(x)-G*\omega^\mu(x^\mu)|\\
\rightarrow&0.
\end{split}
\end{equation}
\end{proof}

\section{Proof of Theorem \ref{29}}

In this section we will give proof of Theorem \ref{29}. Before doing this we need to establish several preliminary lemmas first. We will show that the weakly star limit $\omega^\mu\in\overline{\mathcal{N}^\mu}$ of $\omega_1^\lambda$ actually belongs to $\mathcal{N}^\mu$, $x^\mu\in \bar{D}$ actually in $D$ and $(\omega^\mu,x^\mu)$ is a  weak solution to the stationary vortex-wave system \eqref{24}.

\begin{lemma}\label{28}
Let $\omega\in L^\infty(D), x\in D$, then $(\omega,x)$ is a weak solution of \eqref{24} if the following two conditions are satisfied

(1). For any $y\in D$, $G*\omega(y)-H(y)\leq G*\omega(x)-H(x)$.

(2). For any $v\in \mathcal{R}(\omega)$,
\begin{equation}
E(v)+G*v(x)\leq E(\omega)+G*\omega(x).
\end{equation}

\end{lemma}
\begin{proof}
Condition $(1)$ in Lemma \ref{28} implies that $x$ is a maximum point for the function $G*\omega-H$ in $D$, so $\nabla G*\omega(x)- \nabla H(x)=0$.

In the following, for the sake of convenience  set
\begin{equation}
F(v,y)=E(v)+G*v(y), \,\,v\in \mathcal{R}(\omega),\,y\in D.
\end{equation}
For any given $\phi\in C^{\infty}_0(D)$, define a family of $C^1$ transformations $\Phi_t(x):D \hookrightarrow D$ for $t\in(-\infty,+\infty)$ by the following ordinary differential equation:
\begin{equation}\label{400}
\begin{cases}\frac{d\Phi_t(x)}{dt}=J\nabla\phi(\Phi_t(x)),\,\,\,t\in\mathbb R, \\
\Phi_0(x)=x,
\end{cases}
\end{equation}
where $J$ denotes clockwise rotation through $\frac{\pi}{2}$ as before. Note that $\eqref{400}$ is solvable for all $t$ since $J\nabla\phi$ is a smooth vector field with compact support in $D$. It's easy to see that $J\nabla\phi$ is divergence-free, so by Liouville theorem(see \cite{MPu}, Appendix 1.1) $\Phi_t(x)$ is area-preserving, or equivalently for any measurable set $A\subset D$
\begin{equation}
|\Phi_t(A)|=|A|.
\end{equation}
 Now define a family of test functions
\begin{equation}
\omega^{(t)}(x)\triangleq\omega(\Phi_{-t}(x)).
\end{equation}
Since $\Phi_t$ is area-preserving, we have $\omega^{(t)}\in\mathcal{R}(\omega)$, then condition $(1)$ in Lemma \ref{28} implies that $F(\omega^{(t)},x)$ attains its maximum at $t=0$, so $\frac{d}{dt}F(\omega^{(t)},x)|_{t=0}=0$.
 Expanding $F(\omega^{(t)},x_0)$ at $t=0$ gives
\[\begin{split}
F(\omega^{(t)},x)=&\frac{1}{2}\int_D\int_DG(y,z)\omega(\Phi_{-t}(y))\omega(\Phi_{-t}(z))dydz+\int_DG(x,y)\omega(\Phi_{-t}(y))dy\\
=&\frac{1}{2}\int_D\int_DG(\Phi_{t}(y),\Phi_{t}(z))\omega(y)\omega(z)dydz+\int_DG(x,\Phi_{t}(y))\omega(y)dy\\
=&E(\omega)+t\int_D\omega(y)\nabla (G*\omega(y)+ G(x,y))\cdot J\nabla\phi(y) dy+o(t),
\end{split}\]
as $t\rightarrow 0$. So we have
\[\int_D\omega(y)\nabla (G*\omega(y)+ G(x,y))\cdot J\nabla\phi(y) dy=0,\,\,\forall \phi\in C_c^\infty(D),\]
which completes the proof.
\end{proof}

To apply Lemma \ref{28}, we need more information about $(\omega^\mu,x^\mu)$.

\begin{lemma}\label{41}
$x^\mu\in D$.
\end{lemma}
\begin{proof}
By Lemma \ref{700} and the symmetry of the Green's function,
\begin{equation}\label{703}
\begin{split}
E(\omega^\lambda)=&\frac{1}{2}\int_D\int_DG(x,y)(\omega^\lambda_1+\omega^\lambda_2)(x)(\omega^\lambda_1+\omega^\lambda_2)(y)dxdy\\
=&E(\omega^\lambda_1)+E(\omega^\lambda_2)+\int_D\int_DG(x,y)\omega^\lambda_1(x)\omega^\lambda_2(y)dxdy\\
=&E(\omega^\mu)-\frac{1}{4\pi}\int_D\int_D\ln|x-y|\omega^\lambda_2(x)\omega^\lambda_2(y)dxdy+G*\omega^\mu(x^\mu)-H(x^\mu)+o(1).
\end{split}
\end{equation}
By rearrangement inequality(see \cite{LL}, \S3.4),
\begin{equation}\label{704}
\begin{split}
-\frac{1}{4\pi}\int_D\int_D\ln|x-y|\omega^\lambda_2(x)\omega^\lambda_2(y)dxdy&\leq \sup_{x\in D}-\frac{\lambda}{4\pi}\int_{D}\ln|x-y|\omega^\lambda_2(y)dy\\
&\leq -\frac{\lambda}{4\pi}\int_{B_\varepsilon(0)}\ln|y|dy\\
&\leq -\frac{1}{4\pi}\ln\varepsilon+C.
\end{split}
\end{equation}
So we have
\begin{equation}\label{705}
E(\omega^\lambda)\leq E(\omega^\mu)-\frac{1}{4\pi}\ln\varepsilon+C+G*\omega^\mu(x^\mu)-H(x^\mu)+o(1),
\end{equation}
that is
\begin{equation}\label{706}
E(\omega^\lambda)+\frac{1}{4\pi}\ln\varepsilon\leq E(\omega^\mu)+G*\omega^\mu(x^\mu)-H(x^\mu)+o(1).
\end{equation}
If $x^\mu\in\partial D$, then $H(x^\mu)=+\infty$, which means that $E(\omega^\lambda)+\frac{1}{4\pi}\ln\varepsilon\rightarrow-\infty$ as $\lambda\rightarrow+\infty$, which is a contradiction to Lemma \ref{35}.
\end{proof}

\begin{lemma}\label{42}
$\sup_{v\in \mathcal{N}^\mu}(E(v)+G*v(x^\mu))=\sup_{v\in \overline{\mathcal{N}^\mu}}(E(v)+G*v(x^\mu))$.
\end{lemma}
\begin{proof}
Firstly it is obvious that $\sup_{v\in \mathcal{N}^\mu}(E(v)+G*v(x^\mu))\leq\sup_{v\in \overline{\mathcal{N}^\mu}}(E(v)+G*v(x^\mu))$.

On the other hand, for any ${\omega}\in \overline{\mathcal{N}^\mu}$ we can choose a sequence $\{\omega^n\}\subset\mathcal{N}^\mu$ such that $\omega^n\rightarrow {\omega}$ weakly star in $L^\infty(D)$, then
\[E(\omega^n)+G*\omega^n(x^\mu)\rightarrow E({\omega})+G*{\omega}(x^\mu),\]
which means that $\sup_{v\in \mathcal{N}^\mu}(E(v)+G*v(x^\mu))\geq E({\omega})+G*{\omega}(x^\mu)$. Since ${\omega}\in \overline{\mathcal{N}^\mu}$ is arbitrary, we have
\[\sup_{v\in \mathcal{N}^\mu}(E(v)+G*v(x^\mu))\geq\sup_{v\in \overline{\mathcal{N}^\mu}}(E(v)+G*v(x^\mu)),\]
which completes the proof.
\end{proof}

\begin{lemma}\label{43}
$E(\omega^\mu)+G*\omega^\mu(x^\mu)=\sup_{\omega\in\mathcal{N}^\mu}(E(\omega)+G*\omega(x^\mu))$.
\end{lemma}

\begin{proof}
Recall that $\omega^\lambda_2=\lambda I_{B^\lambda}$.
By choosing $v^\lambda=v^\lambda_1+\omega^\lambda_2$, such that $v^\lambda_1\in\mathcal {N}^\mu, v^\lambda_1\equiv0$ a.e. on $B^\lambda$, it is obvious that $v^\lambda\in\mathcal{M}^\lambda$. As a consequence we have $E(\omega^\lambda)\geq E(v^\lambda)$, that is,
\begin{equation}
\begin{split}
E(\omega_1^\lambda)+E(\omega_2^\lambda)+\int_DG*\omega_1^\lambda(x)\omega^\lambda_2(x)dx\geq E(v_1^\lambda)+E(\omega_2^\lambda)+\int_DG*v_1^\lambda(x)\omega^\lambda_2(x)dx,
\end{split}
\end{equation}
which gives
\begin{equation}
E(\omega_1^\lambda)+\int_DG*\omega_1^\lambda(x)\omega^\lambda_2(x)dx\geq E(v_1^\lambda)+\int_DG*v_1^\lambda(x)\omega^\lambda_2(x)dx.
\end{equation}
By Lemma \ref{700} it follows

\begin{equation}
E(\omega^\mu)+G*\omega^\mu(x^\mu)\geq E(v_1^\lambda)+G*v_1^\lambda(x^\mu)+o(1).
\end{equation}
Since $diam(supp(\omega^\lambda_2))\rightarrow 0$ and $E$ is a continuous functional on $\mathcal{N}^\mu$, $v^\lambda_1$ can be any element in $\mathcal{N}^\mu$ as $\lambda\rightarrow+\infty$, that is
\begin{equation}
E(\omega^\mu)+G*\omega^\mu(x^\mu)\geq  E(v)+G*v(x^\mu),\,\,\forall v\in\mathcal{N}^\mu,
\end{equation}
which, combined with Lemma \ref{42} leads to the desired result.
\end{proof}

\begin{lemma}\label{44}
$\omega^\mu\in \mathcal{N}^\mu$ and $\omega^\mu=\mu I_{\{G*\omega^\mu+G(x^\mu,\cdot)>b^\mu\}}$ for some $b^\mu>0$.
\end{lemma}

\begin{proof}

Define $\mathcal{F}=\{\omega\in L^\infty(D)\,|\,1\leq\omega\leq \mu,\int_D\omega(x)dx=1\},$ then for $\mathcal{F}$ we
have the following two claims.

Claim 1: $\overline{\mathcal{N}^\mu}\subset\mathcal{F}$.

Proof of Claim 1: By the definition of $\overline{\mathcal{N}^\mu}$ it suffices to show that $\mathcal{F}$ is closed in the weak star topology in $L^\infty(D)$. Let $\omega^n\in \mathcal{F}$, $\omega^n\rightarrow \omega^*$ weakly star in $L^\infty(D)$, that is,

\begin{equation}
\lim_{n\rightarrow +\infty}\int_D\omega^n(x)\phi(x)dx=\int_D\omega^*(x)\phi(x)dx,\,\,\forall \phi\in L^1(D),
\end{equation}
it suffices to show that $\omega^*\in\mathcal{F}$.

Firstly by choosing $\phi(x)\equiv1$ we have
\[\lim_{n\rightarrow +\infty}\int_D\omega^n(x)dx=\int_D\omega^*(x)dx=1.\]

Now we prove $0\leq \omega^*\leq \mu$ by contradiction. Suppose that $|\{\omega^*>\mu\}|>0$, then there exists $\varepsilon_0>0$ such that $|\{\omega^*\geq \mu+\varepsilon_0\}|>0$. Denote $A=\{\omega^*\geq \mu+\varepsilon_0\}$, then for $\phi=I_A$ we have
\[0=\lim_{n\rightarrow +\infty}\int_D(\omega^*-\omega^n)(x)\phi(x)dx=\lim_{n\rightarrow +\infty}\int_{A}\omega^*(x)-\omega^n(x)dx.\]
On the other hand
\[\lim_{n\rightarrow +\infty}\int_{A}(\omega^*-\omega^n)(x)dx\geq\varepsilon_0|A|>0,\]
which is a contradiction. So we have $\omega^*\leq \mu$ a.e. on $D$.

Lastly, a similar argument suggests $\omega^*\geq 0$ a.e. on $D$, which completes the proof of Claim 1.

Claim 2: There exists $\tilde{\omega}\in\mathcal{F}$ such that $E(\tilde{\omega})+G*\tilde{\omega}(x^\mu)=\sup_{\omega\in\mathcal{F}}E(\omega)+G*\omega(x^\mu)$, moreover, any maximizer $\tilde{\omega}$ has the form $\tilde{\omega}=\mu\lambda_{\{G*\tilde{\omega}+G(x^\mu,\cdot)>b^\mu\}}$ for some $b^\mu>0$.

Proof of Claim 2: Firstly we show that $\sup_{\omega\in\mathcal{F}}E(\omega)+G*\omega(x^\mu)<+\infty$. In fact, for any $\omega\in\mathcal{F}$,
\begin{equation}
\begin{split}
E(\omega)+G*\omega(x^\mu)=&\frac{1}{2}\int_D\int_DG(x,y)\omega(x)\omega(y)dxdy+\int_DG(x^\mu,y)\omega(y)dy\\
\leq&\frac{\mu^2}{2}\int_D\int_D|G(x,y)|dxdy+\mu\int_DG|(x^\mu,y)|dy\\
\leq&C,
\end{split}
\end{equation}
where $C$ is a positive number not depending on $\omega$(may depending on $\mu$). Now we choose $\omega^n\in\mathcal{F}$ such that $\omega^n\rightarrow\tilde{\omega}$ and $E(\omega^n)+G*\omega^n(x^\mu)\rightarrow \sup_{\omega\in\mathcal{F}}E(\omega)+G*\omega(x^\mu)$. An argument similar to the one used in Lemma \ref{700} gives
\begin{equation}
E(\tilde{\omega})+G*\tilde{\omega}(x^\mu)=\sup_{\omega\in\mathcal{F}}(E(\omega)+G*\omega(x^\mu)).
\end{equation}

Now we prove that $\tilde{\omega}$ is a vortex patch with the form $\tilde{\omega}=\mu\lambda_{\{G*\tilde{\omega}+G(x^\mu,\cdot)>b^\mu\}}$ for some $b^\mu>0$. Define a family of test functions $\omega^{(s)}(x)=\tilde{\omega}+s[z_0(x)-z_1(x)]$, $s>0$, where $z_0,z_1$ satisfies
\begin{equation}
\begin{cases}
z_0,z_1\in L^\infty(D),\, z_0,z_1\geq 0,\, \int_Dz_0dx=\int_D z_1dx,

 \\z_0=0\text{\,\,\,\,\,\,} in\text{\,\,} D\verb|\|\{\tilde{\omega}\leq \mu-\delta\},
 \\z_1=0\text{\,\,\,\,\,\,} in\text{\,\,} D\verb|\|\{\tilde{\omega}\geq\delta\},
\end{cases}
\end{equation}
here $\delta$ is any positive number. Note that for fixed $z_0,z_1$ and $\delta$, $\omega^{(s)}\in \mathcal{F}$ provided $s$ is sufficiently small(depending on $\delta,z_0,z_1$). So we have
\begin{equation}
\frac{d}{ds}[E(\omega^{(s)})+G*\omega^{(s)}(x^\mu)]\Big|_{s=0^+}\leq 0,
\end{equation}
 which gives

\begin{equation}\label{671}
\sup_{\{\tilde{\omega}<\mu\}}(G*\tilde{\omega}+G(x^\mu,\cdot))\leq\inf_{\{\tilde{\omega}>0\}}(G*\tilde{\omega}+G(x^\mu,\cdot)).
\end{equation}
Now it is obvious that there exists $r>0$ such that $\tilde{\omega}\equiv \mu$ a.e. in $B_r(x^\mu)$(otherwise the left hand side of \eqref{671} equals $+\infty$ ). Moreover, we can choose $r$ sufficiently small such that
\begin{equation}\label{777}
\inf_{\{\tilde{\omega}>0\}}(G*\tilde{\omega}+G(x^\mu,\cdot))=\inf_{\{\tilde{\omega}>0\}\cap D_r}(G*\tilde{\omega}+G(x^\mu,\cdot)),
\end{equation}
where $D_r=D\setminus \overline{B_r(x^\mu)}$. Then we have
\begin{equation}
\sup_{\{\tilde{\omega}<\mu\}\cap D_r}(G*\tilde{\omega}+G(x^\mu,\cdot))\leq\inf_{\{\tilde{\omega}>0\}\cap D_r}(G*\tilde{\omega}+G(x^\mu,\cdot)).
\end{equation}
Since $\overline{D_r}$ is connected (for sufficiently small $r$) and $\overline{\{\tilde{\omega}<\mu\}\cap D_r}\cup\overline{\{\tilde{\omega}>0\}\cap D_r}=\overline{D_r}$, we have $\overline{\{\tilde{\omega}<\mu\}\cap D_r}\cap\overline{\{\tilde{\omega}>0\}\cap D_r}\neq\varnothing$, then by the continuity of $G*\tilde{\omega}+G(x^\mu,\cdot)$ on $\overline{D_r}$,
\begin{equation}
\sup_{\{\tilde{\omega}<\mu\}\cap D_r}(G*\tilde{\omega}+G(x^\mu,\cdot))=\inf_{\{\tilde{\omega}>0\}\cap D_r}(G*\tilde{\omega}+G(x^\mu,\cdot)).
\end{equation}
Now define
\begin{equation}
b^\mu=\sup_{\{\tilde{\omega}<\mu\}\cap D_r}(G*\tilde{\omega}+G(x^\mu,\cdot))=\inf_{\{\tilde{\omega}>0\}\cap D_r}(G*\tilde{\omega}+G(x^\mu,\cdot)),
\end{equation}
by maximum principle it is easy to see that $\mu>0$, and it is also obvious that
\begin{equation}
\begin{cases}
\tilde{\omega}=0\text{\,\,\,\,\,\,$a.e.$\,} in\text{\,\,}\{G*\tilde{\omega}+G(x^\mu,\cdot)<b^\mu\}\cap D_r,
 \\ \tilde{\omega}=\mu\text{\,\,\,\,\,\,$a.e.$\,} in\text{\,\,}\{G*\tilde{\omega}+G(x^\mu,\cdot)>b^\mu\}\cap D_r.
\end{cases}
\end{equation}
On $\{G*\tilde{\omega}+G(x^\mu,\cdot)=b^\mu\}\cap D_r$, we have $\nabla(G*\tilde{\omega}+G(x^\mu,\cdot))=0\text{\,\,} a.e.$, which gives $\tilde{\omega}=-\Delta (G*\tilde{\omega})=-\Delta (G*\tilde{\omega}+G(x^\mu,\cdot))=0$.
Now it remains to show that $G*\tilde{\omega}+G(x^\mu,\cdot)>b^\mu$ on $B_r(x^\mu).$ This is an easy consequence of the maximum principle. In fact, by \eqref{777}

\begin{equation}
\begin{split}
b^\mu&=\inf_{\{\tilde{\omega}>0\}\cap D_r}(G*\tilde{\omega}+G(x^\mu,\cdot)),\\
&=\inf_{\{\tilde{\omega}>0\}}(G*\tilde{\omega}+G(x^\mu,\cdot)),\\
&\leq\inf_{ B_r(x^\mu)}(G*\tilde{\omega}+G(x^\mu,\cdot))\\
&\leq \inf_{\partial B_r(x^\mu)}(G*\tilde{\omega}+G(x^\mu,\cdot)),
\end{split}
\end{equation}
then by strong maximum principle we have $G*\tilde{\omega}+G(x^\mu,\cdot)>b^\mu$ on $B_r(x^\mu)$.

In conclusion, we have proved that $\tilde{\omega}$ has the form $\tilde{\omega}=\mu I_{\{G*\tilde{\omega}+G(x^\mu,\cdot)>b^\mu\}}$ for some $b^\mu>0$, which completes the proof of Claim 2.

Now we proceed to prove Lemma \ref{44}. By Claim 2 it is easy to see that
\begin{equation}
\sup_{\omega\in\mathcal{N}^\mu}(E(\omega)+G*\omega(x^\mu))=\sup_{\omega\in\mathcal{F}}(E(\omega)+G*\omega(x^\mu)),
\end{equation}
therefore we obtain
\[E(\omega^\mu)+G*\omega^\mu(x^\mu)=\sup_{\omega\in\mathcal{F}}(E(\omega)+G*\omega(x^\mu)).\]
Using Claim 2 again we get the desired result.
\end{proof}

\begin{remark}
 Lemma \ref{44} is essential to this paper. We remark that Corollary 3.4 in \cite{B4} can not be applied here anymore since $\nabla^2G$ is not a locally integrable function. The proof we give here is based on the idea of Turkington in \cite{T} with some modifications.
\end{remark}

Now we are ready to prove Theorem \ref{29}.
\begin{proof}[Proof of Theorem \ref{29}]
Firstly by Lemma \ref{43} and Lemma \ref{44}, $\omega^\mu$ satisfies $(2)$ in Lemma \ref{28} and has the form $\omega^\mu=\mu I_{\{G*\omega^\mu+G(x^\mu,\cdot)>b^\mu\}}$ for some $b^\mu>0$. It suffices to show that $x^\mu$ satisfies $(1)$ in Lemma \ref{28}.

Fix $x_1\in D$ and define $v^\lambda=v^\lambda_1+v^\lambda_2$, where $v^\lambda_2=\lambda I_{B_\varepsilon(x_1)}$, $v^\lambda_1\in\mathcal{N}^\mu$ and $v^\lambda_1=0$ a.e. on $B_\varepsilon(x_1)$. It is easy to check that $v^\lambda\in\mathcal{M}^\lambda$, so we have $E(\omega^\lambda)\geq E(v^\lambda)$, that is,
\begin{equation}
E(\omega^\lambda_1)+E(\omega^\lambda_2)+\int_DG*\omega^\lambda_1(x)\omega^\lambda_2(x)dx\geq E(v^\lambda_1)+E(v^\lambda_2)+\int_DG*v^\lambda_1(x)v^\lambda_2(x)dx,
\end{equation}
then by Lemma \ref{700} for $\lambda$ sufficiently large we have

\begin{equation}
\begin{split}
&E(\omega^\mu)-\frac{1}{4\pi}\int_D\int_D\ln|x-y|\omega^\lambda_2(x)\omega^\lambda_2(y)dxdy-H(x^\mu)+G*\omega^\mu(x^\mu)+o(1)\\
&\geq E(v^\lambda_1)-\frac{1}{4\pi}\int_D\int_D\ln|x-y|v^\lambda_2(x)v^\lambda_2(y)dxdy-H(x_1)+G*v^\lambda_1(x_1).
\end{split}
\end{equation}
On the other hand, by Riesz's rearrangement inequality(see \cite{LL}, \S3.7),
\begin{equation}
-\frac{1}{4\pi}\int_D\int_D\ln|x-y|\omega^\lambda_2(x)\omega^\lambda_2(y)dxdy\leq -\frac{1}{4\pi}\int_D\int_D\ln|x-y|v^\lambda_2(x)v^\lambda_2(y)dxdy.
\end{equation}
So we have
\begin{equation}
E(\omega^\mu)-H(x^\mu)+G*\omega^\mu(x^\mu)+o(1)
\geq E(v^\lambda_1)-H(x_1)+G*v^\lambda_1(x_1).
\end{equation}
Again, since $E$ is a continuous functional on $\mathcal{N}^\mu$ and $|B_\varepsilon(x_1)|\rightarrow0$, $v^\lambda_1$ can be any element in $\mathcal{N}^\mu$ as $\lambda\rightarrow+\infty$, that is,
\begin{equation}
E(\omega^\mu)-H(x^\mu)+G*\omega^\mu(x^\mu)
\geq E(v)-H(x_1)+G*v(x_1),\,\,\forall v\in\mathcal{N}^\mu.
\end{equation}
Especially we can choose $v=\omega^\mu$, then it follows
\begin{equation}
-H(x^\mu)+G*\omega^\mu(x^\mu)
\geq -H(x_1)+G*\omega^\mu(x_1),\,\,\forall x_1\in D,
\end{equation}
which means that $x^\mu$ satisfies $(1)$ in Lemma \ref{28}. Therefore we complete the proof.

\end{proof}

\section{Proof of Theorem \ref{90}}

Up to now we have constructed $(\omega^\mu,x^\mu)$ as a steady vortex patch solution to the vortex-wave system for fixed $\mu$. Now we consider the asymptotic behavior of $(\omega^\mu,x^\mu)$ when $\mu\rightarrow +\infty$. As has been stated in Theorem \ref{90}, we will show that both the support of $\omega^\mu$ and $x^\mu$ converge to a minimum point of $H$, which is a stationary solution to the Kirchhoff-Routh equation.

In this section we shall use $C$ to denote various positive numbers independent of $\mu$.
 Theorem \ref{90} is an easy consequence of the following several lemmas.

\begin{lemma}\label{51}
For any $\omega\in\mathcal{N}^\mu, x\in D$, we have \[E(\omega)+G*\omega(x)-H(x)\leq E(\omega^\mu)+G*\omega^\mu(x^\mu)-H(x^\mu).\]
\end{lemma}
\begin{proof}
For fixed $x\in D$, define a family of test functions $v^\lambda=v^\lambda_1+v^\lambda_2$, $v^\lambda_2=\lambda I_{B_\varepsilon(x)}$, $v^\lambda_1\in \mathcal{N}^\mu$ and $v^\lambda_1=0$ a.e. on $B_\varepsilon(x)$. It is easy to check $v^\lambda\in \mathcal{M}^\lambda$, then by definition
$E(v^\lambda)\leq E(\omega^\lambda)$, that is,

\begin{equation}
E(v^\lambda_1)+E(v^\lambda_2)+\int_DG*v^\lambda_1(y)v^\lambda_2(y)dy\leq E(\omega^\lambda_1)+E(\omega^\lambda_2)+\int_DG*\omega^\lambda_1(y)\omega^\lambda_2(y)dy,
\end{equation}
again by Lemma \ref{700}
\begin{equation}\label{22222}
\begin{split}
&E(v^\lambda_1)-\frac{1}{4\pi}\int_D\int_D\ln|y-z|v^\lambda_2(y)v^\lambda_2(z)dydz-H(x)+G*v^\lambda_1(x)\\
\leq &E(\omega^\mu)-\frac{1}{4\pi}\int_D\int_D\ln|y-z|\omega^\lambda_2(y)\omega^\lambda_2(z)dydz-H(x^\mu)+G*\omega^\mu(x^\mu)+o(1),
\end{split}
\end{equation}
where $o(1)\rightarrow0$ as $\lambda\rightarrow+\infty$.
Using Riesz's rearrangement inequality, from \eqref{22222}
 we have
\begin{equation}
E(v^\lambda_1)-H(x)+G*v^\lambda_1(x)\leq E(\omega^\mu)-H(x^\mu)+G*\omega^\mu(x^\mu)+o(1).
\end{equation}
As $\lambda\rightarrow+\infty, v^\lambda_1$ can be any element in $\mathcal{N}^\mu$, so we obtain
\begin{equation}
E(\omega)-H(x)+G*\omega(x)\leq E(\omega^\mu)-H(x^\mu)+G*\omega^\mu(x^\mu),\,\,\text{for all}\, (\omega,x)\in (\mathcal{N}^\mu, D).
\end{equation}
\end{proof}

\begin{remark}
One can also maximize $E(\omega)+G*\omega(x)-H(x)$ for $(\omega,x)\in (\mathcal{N}^\mu, D)$ to obtain steady solution to the vortex-wave system, but it is much more interesting to construct solutions from the Euler equation, because the vortex-wave itself is an approximation of the Euler equation when a part of the vorticity is sufficiently concentrated.
\end{remark}

In the following $s$ will be the positive number defined by $\mu\pi s^2=1$.

\begin{lemma}
There exists $\delta_0>0,$ not depending on $\mu$, such that $dist(x^\mu,\partial D)>\delta_0$.
\end{lemma}
\begin{proof}
Fix $x_1\in D$ and define $\bar{\omega}^\mu=\mu I_{B_s(x_1)}$, then $\bar{\omega}^\mu\in\mathcal{N}^\mu$, by Lemma \ref{51}
\begin{equation}
E(\bar{\omega}^\mu)+G*\bar{\omega}^\mu(x_1)-H(x_1)\leq E(\omega^\mu)+G*\omega^\mu(x^\mu)-H(x^\mu).
\end{equation}
Using Riesz's rearrangement inequality we get
\begin{equation}
\begin{split}
&-H(x_1)-2H(x_1)-H(x_1)+o(1)\\
\leq& -\frac{1}{2}\int_D\int_Dh(x,y)\omega^\mu(x)\omega^\mu(y)dxdy-\int_Dh(x^\mu,y)\omega^\mu(y)dy-H(x^\mu),
\end{split}
\end{equation}
since $h$ is bounded from below in $D\times D$, we have
\begin{equation}
H(x^\mu)\leq C,
\end{equation}
then we get the desired result by the fact $\lim_{x\rightarrow\partial D}H(x)=+\infty$.
\end{proof}

\begin{lemma}\label{59}
$G*\omega^\mu(x^\mu)\geq -\frac{1}{2\pi}\ln s-C$.
\end{lemma}
\begin{proof}
Since $dist(x^\mu,\partial D)>\delta_0$, we can define $\bar{\omega}^\mu=\mu I_{B_s(x^\mu)}\in \mathcal{N}^\mu$, then by Lemma \ref{43}
\begin{equation}
E(\bar{\omega}^\mu)+G*\bar{\omega}^\mu(x^\mu)\leq E(\omega^\mu)+G*\omega^\mu(x^\mu).
\end{equation}
Again by Riesz's rearrangement inequality we have
\begin{equation}\begin{split}
&-\frac{1}{2}\int_D\int_Dh(x,y)\bar{\omega}^\mu(x)\bar{\omega}^\mu(y)dxdy-\frac{1}{2\pi}\int_D\ln|x^\mu-y|\bar{\omega}^\mu(y)dy-\int_Dh(x^\mu,y)\bar{\omega}^\mu(y)dy\\
\leq& -\frac{1}{2}\int_D\int_Dh(x,y){\omega}^\mu(x){\omega}^\mu(y)dxdy+G*\omega^\mu(x^\mu),
\end{split}
\end{equation}
since $h$ is bounded from below in $D\times D$ and $x^\mu$ is away from $\partial D$, we get
\begin{equation}\begin{split}
G*\omega^\mu(x^\mu)&\geq -C-H(x^\mu)-\frac{1}{2\pi}\int_D\ln|x^\mu-y|\bar{\omega}^\mu(y)dy-2H(x^\mu)\\
&\geq -\frac{\mu}{2\pi}\int_{B_s(0)}\ln|y|dy-C\\
&\geq -\frac{1}{2\pi}\ln s-C,
\end{split}
\end{equation}
where we use $\int_{B_s(0)}\ln|y|dy=\pi s^2(\ln s-\frac{1}{2})$.
\end{proof}

\begin{lemma}\label{88}
There exists $\rho^\mu$ satisfying $\rho^\mu\rightarrow0$ and $\int_{B_{\rho^\mu}(x^\mu)}\omega^\mu(x)dx\rightarrow 1$ as $\mu\rightarrow+\infty$.
\end{lemma}
\begin{proof}
By Lemma \ref{59},
\begin{equation}
-\frac{1}{2\pi}\int_D\ln|x^\mu-y|\omega^\mu(y)dy-\int_Dh(x^\mu,y)\omega^\mu(y)dy\geq -\frac{1}{2\pi}\ln s-C,
\end{equation}
since $h$ is bounded from below in $D\times D$, we get
\begin{equation}
\int_D\ln\frac{s}{|x^\mu-y|}\omega^\mu(y)dy\geq -C.
\end{equation}
Now choose $R>1$ to be determined, we have
\begin{equation}
\int_{B_{Rs}(x^a)}\ln\frac{s}{|x^\mu-y|}\omega^\mu(y)dy+\int_{D\setminus B_{Rs}(x^\mu)}\ln\frac{s}{|x^\mu-y|}\omega^\mu(y)dy\geq -C.
\end{equation}
Observe that
\begin{equation}
\int_{B_{Rs}(x^\mu)}\ln\frac{s}{|x^\mu-y|}\omega^\mu(y)dy\leq  \mu\int_{B_{s}(x^\mu)}\ln\frac{s}{|x^\mu-y|}dy=\frac{1}{2},
\end{equation}
so we get
\begin{equation}
\int_{D\setminus B_{Rs}(x^\mu)}\ln\frac{s}{Rs}\omega^\mu(y)dy\geq\int_{D\setminus B_{Rs}(x^\mu)}\ln\frac{s}{|x^\mu-y|}\omega^\mu(y)dy\geq -C,
\end{equation}
which gives
\begin{equation}
\int_{D\setminus B_{Rs}(x^\mu)}\omega^\mu(y)dy\leq\frac{C}{\ln R},
\end{equation}
but $\int_D\omega^\mu(x)dx=1$, we have
\begin{equation}
1\geq\int_{ B_{Rs}(x^\mu)}\omega^\mu(y)dy\geq 1-\frac{C}{\ln R},
\end{equation}
then the lemma is proved by choosing $R=s^{-\frac{1}{2}}$ and $\rho^\mu=s^\frac{1}{2}$.
\end{proof}

Since ${x^\mu}$ is bounded and away from $\partial D$, we assume that $x^\mu\rightarrow x^*\in D$(up to a subsequence) as $\mu\rightarrow+\infty$. An argument similar to the one in Remark \ref{457} shows that $\omega^\mu\rightarrow\delta(x^*)$ in the distributional sense.

\begin{lemma}\label{1000}
$H(x^*)=\min_{x\in D}H(x).$
\end{lemma}
\begin{proof}
Since $H=+\infty$ on $\partial D$, there exists $x_1$ such that $H(x_1)=\min_{x\in D}H(x)$. It suffices to show $H(x_1)\geq H(x^*)$.
Define $\bar{\omega}^\mu=\mu I_{B_s(x_1)}\in \mathcal{N}^\mu$, then by Lemma \ref{51}
\begin{equation}
E(\bar{\omega}^\mu)+G*\bar{\omega}^\mu(x_1)-H(x_1)\leq E(\omega^\mu)+G*\omega^\mu(x^\mu)-H(x^\mu),
\end{equation}
that is,
\begin{equation}\label{444}
\begin{split}
&-\frac{1}{4\pi}\int_D\int_D\ln|x-y|\bar{\omega}^\mu(x)\bar{\omega}^\mu(y)dxdy-\frac{1}{2}\int_D\int_Dh(x,y)\bar{\omega}^\mu(x)\bar{\omega}^\mu(y)dxdy\\
&-\frac{1}{2\pi}\int_D\ln|x_1-y|\bar{\omega}^\mu(y)dy-\int_Dh(x_1,y)\bar{\omega}^\mu(y)dy-H(x_1)\\
\leq&-\frac{1}{4\pi}\int_D\int_D\ln|x-y|{\omega}^\mu(x){\omega}^\mu(y)dxdy-\frac{1}{2}\int_D\int_Dh(x,y){\omega}^\mu(x){\omega}^\mu(y)dxdy\\
&-\frac{1}{2\pi}\int_D\ln|x^\mu-y|{\omega}^\mu(y)dy-\int_Dh(x^\mu,y){\omega}^\mu(y)dy-H(x^\mu)
\end{split}
\end{equation}
Taking the limit in \eqref{444}, by rearrangement inequality we obtain
\begin{equation}
H(x_1)\geq H(x^*),
\end{equation}
which completes the proof.

\end{proof}
\begin{proof}[Proof of Theorem \ref{90}]
By choosing $r^\mu=\rho^\mu+|x^\mu-x^*|$, Theorem \ref{90} is an easy consequence of Lemma \ref{88} and Lemma \ref{1000}.
\end{proof}
\begin{remark}
There may be a better convergence for $\omega^\mu$, that is, the support of $\omega^\mu$ shrinks to $x^*$ as $\mu\rightarrow+\infty$, but we have not yet proved this. The main difficulty to estimate the size of $supp(\omega^\mu)$ is that the mutual interaction energy between the background vorticity and the point vortex is very large, and energy estimate does not provide enough information anymore.
\end{remark}

\smallskip

{\bf Acknowledgements:}
{\it D. Cao was supported by NNSF of China (grant No. 11331010) and Chinese Academy of Sciences by
 grant QYZDJ-SSW-SYS021. G. Wang was supported by  NNSF of China (grant No.11771469).

}

\renewcommand\refname{References}
\renewenvironment{thebibliography}[1]{%
\section*{\refname}
\list{{\arabic{enumi}}}{\def\makelabel##1{\hss{##1}}\topsep=0mm
\parsep=0mm
\partopsep=0mm\itemsep=0mm
\labelsep=1ex\itemindent=0mm
\settowidth\labelwidth{\small[#1]}%
\leftmargin\labelwidth \advance\leftmargin\labelsep
\advance\leftmargin -\itemindent
\usecounter{enumi}}\small
\def\newblock{\ }
\sloppy\clubpenalty4000\widowpenalty4000
\sfcode`\.=1000\relax}{\endlist}
\bibliographystyle{model1b-num-names}

\end{document}